\documentclass[12pt, reqno]{amsart}
\usepackage{mathrsfs}
\usepackage{graphicx}
\usepackage{tikz}
\usepackage{amsmath}
\usepackage{amsfonts}
\usepackage{amssymb}
\usepackage[pagebackref]{hyperref}
\hypersetup{backref, pagebackref, colorlinks=true}
\usepackage{cite}
\usepackage{color}

\usepackage{epic}

\newtheorem{thm}{Theorem}[section]

\newtheorem{lem}[thm]{Lemma}

\newtheorem{defi}{Definition}

\usepackage{amsthm,amsmath,amssymb,url,cite,color}

\def\blue{\textcolor{blue}}
\def\red{\textcolor{red}}
\def\magenta{\textcolor{magenta}}

\numberwithin{equation}{section}

\def\da{\mathrm{da}}
\def\dd{\mathrm{dd}}

\def\M{\mathrm{M}}

\def\V{\mathrm{V}}

\def\red{\mathrm{red}}

\def\ln{\mathrm{ln}}
\def\des{\mathrm{des}}
\def\asc{\mathrm{asc}}

\def\asc{\mathrm{asc}}
\def\lmin{\mathrm{lmi}}
\def\rmin{\mathrm{rmi}}
\def\RMI{\mathrm{RMI}}
\def\LMI{\mathrm{LMI}}
\def\lma{\mathrm{lma}}
\def\rma{\mathrm{rma}}

\def\S{\mathfrak{S}}

\def\P{\mathcal{P}}
\def\L{\mathcal{L}}

\def\Orb{\mathrm{Orb}}
\def\inv{\mathrm{inv}}

\def\a{\alpha}
\def\b{\beta}

\def\Z{\mathbb Z}

\def\exc{\mathrm{exc}}
\def\cyc{\mathrm{cyc}}
\def\B{\mathfrak{B}}
\def\maj{\mathrm{maj}}
\def\ai{\mathrm{ai}}

\begin{document}

\title[A $q$-analog of the Stirling-Eulerian Polynomials]{A $q$-analog of the Stirling-Eulerian Polynomials}

\author[Y. Dong]{Yao Dong}
\address[Yao Dong]{Research Center for Mathematics and Interdisciplinary Sciences, Shandong University \& Frontiers Science Center for Nonlinear Expectations, Ministry of Education, Qingdao 266237, P.R. China}
\email{y.dong@mail.sdu.edu.cn}

\author[Z. Lin]{Zhicong Lin}
\address[Zhicong Lin]{Research Center for Mathematics and Interdisciplinary Sciences, Shandong University \& Frontiers Science Center for Nonlinear Expectations, Ministry of Education, Qingdao 266237, P.R. China}
\email{linz@sdu.edu.cn}

\author[Q.Q. Pan]{Qiongqiong Pan}
\address[Qiongqiong Pan]{College of Mathematics and Physics, Wenzhou University, Wenzhou 325035, P.R. China}
\email{qpan@wzu.edu.cn}

\date{\today}
\begin{abstract}
In 1974, Carlitz and Scoville introduced the Stirling-Eulerian polynomial $A_n(x,y|\a,\b)$ as the enumerator of permutations by descents, ascents, left-to-right maxima and right-to-left maxima. Recently, Ji considered a refinement  of $A_n(x,y|\a,\b)$, denoted $P_n(u_1,u_2,u_3,u_4|\a,\b)$, which is the enumerator of permutations by valleys, peaks, double ascents, double descents, left-to-right maxima and right-to-left maxima. Using Chen's context-free grammar calculus, Ji proved a formula for the  generating function of $P_n(u_1,u_2,u_3,u_4|\a,\b)$, generalizing the work of Carlitz and Scoville. Ji's formula has many nice  consequences, one of which is an intriguing  $\gamma$-positivity expansion for $A_n(x,y|\a,\b)$. In this paper, we prove a $q$-analog of Ji's formula by using Gessel's $q$-compositional formula and provide a combinatorial approach to her $\gamma$-positivity expansion of $A_n(x,y|\a,\b)$.
\end{abstract}

\maketitle

\section{Introduction}
 Using Chen's context-free grammar~\cite{Chen}, Ji~\cite{Ji} recently proved a formula for the  generating function of permutations by valleys, peaks, double ascents, double descents, left-to-right maxima and right-to-left maxima. The main objective of this paper is to prove a $q$-analog of Ji's formula by using Gessel's $q$-compositional formula~\cite{Ge}.
We need some further notations and
definitions before we can state our $q$-analog of Ji's formula.

Let $\S_{n}$ be the set of permutations of $[n]:=\{1, \ldots, n\}$. An index $i\in [n-1]$ is a {\em descent} (resp.,~{\em ascent}) of $\sigma=\sigma_1\cdots\sigma_n\in\S_n $ if $\sigma_i>\sigma_{i+1}$ (resp.,~$\sigma_i<\sigma_{i+1}$).
The {\em Eulerian polynomials} $A_n(x)$ can be described as the enumerator of $\S_n$ by descents or ascents, that is
\begin{align}\label{Eulerian}
A_n(x):=\sum_{\sigma \in \S_n}x^{\des(\sigma)}=\sum_{\sigma \in \S_n}x^{\asc(\sigma)},
\end{align}
where  $\des(\sigma)$ (resp.,~$\asc(\sigma)$) denotes the number of descents (resp.,~ascents) of $\sigma$. 
There is a well-known formula (see~\cite[p.~13]{Pet}) for the exponential generating function for $A_n(x)$:
 \begin{equation}\label{euler}
 1+\sum_{n\geq1} xA_n(x)\frac{t^n}{n!}=\frac{1-x}{1-xe^{(1-x)t}}. 
 \end{equation}
 A permutation statistic whose generating function equals~\eqref{Eulerian} is called {\em Eulerian}. Another classical Eulerian statistic (see~\cite[p.~23]{St})  is the number of excedances ``$\exc$'', where  
 $$\exc(\sigma):=\#\{i\in[n-1]: \sigma_i>i\}.$$

Another class of fundamental permutation statistics is the so-called Stirling statistics.  A {\em left-to-right maximum} (resp.,~{\em left-to-right minimum}) of $\sigma\in\S_n$ is an element $\sigma_i$ such that $\sigma_j<\sigma_i$ (resp.,~$\sigma_j>\sigma_i$) for every $j<i$.  Let $\lma(\sigma)$ and $\lmin(\sigma)$ denote the numbers of  left-to-right maxima and left-to-right minima  of $\sigma$, respectively.  
Similarly, we can define the  {\em right-to-left maximum} (resp., {\em right-to-left minimum}) of $\sigma$ and denote $\rma(\sigma)$ (resp., $\rmin(\sigma)$) the number of   right-to-left maxima (resp., right-to-left minima)  of $\sigma$. 
It is well known (see~\cite[p.~20]{St}) that the unsigned {\em Stirling numbers of the first kind} count permutations by their number of left-to-right maxima, namely
 $$
(x)_n:=x(x+1)(x+2)\cdots(x+n-1)=\sum_{\sigma \in \S_n}x^{\lma(\sigma)}. 
$$
Therefore, each statistic whose distribution over $\S_n$ gives $(x)_n$ is called {\em Stirling}. So  ``$\lmin$'', ``$\rma$'' and ``$\rmin$'' are all Stirling statistics by the symmetry of permutations and according to {\em Foata's first fundamental transformation} (see~\cite[Prop.~1.3.1]{St}), the statistic ``$\cyc$'' is also  Stirling, where $\cyc(\sigma)$ denotes the number of cycles of $\sigma$. 
 
 Carlitz and Scoville~\cite{CS} considered  the following polynomials involving Stirling-Eulerian Statistics, which we refer to as the {\em Stirling-Eulerian polynomials}\footnote{These polynomials are called the $(\alpha, \beta)$-Eulerian polynomials by Ji~\cite{Ji}.}:
\begin{align*}
A_n(x,y|\alpha,\beta)&=\sum_{\sigma\in\S_{n}}x^{\mathrm{asc}(\sigma)}y^{\mathrm{des}(\sigma)}\alpha^{\lma(\sigma)-1}\beta^{\rma(\sigma)-1}\\
&=\sum_{\sigma\in\S_{n}}x^{\des(\sigma)}y^{\asc(\sigma)}\alpha^{\lmin(\sigma)-1}\beta^{\rmin(\sigma)-1}.
\end{align*}
 Here the second interpretation follows from the complement 
 $$\sigma=\sigma_1\sigma_2\cdots\sigma_n\mapsto \sigma^c=(n+1-\sigma_1){(n+1-\sigma_2)}\cdots(n+1-\sigma_n)$$ 
 of permutations. In~\cite[Theorem~9]{CS}, they obtained the following generating function formula for $A_n(x,y|\alpha,\beta)$:
 \begin{equation}\label{eq:Carlitz}
 \sum_{n\geq0}A_{n+1}(x,y|\a,\b)\frac{t^n}{n!}=(1+xF(x,y;t))^{\a}(1+yF(x,y;t))^{\b},
 \end{equation} 
 where 
 \begin{equation}\label{def:Fxy}
 F(x,y;t):=\frac{e^{xt}-e^{yt}}{xe^{yt}-ye^{xt}}.
 \end{equation}
 
 Recently, Ji~\cite{Ji} refined the Carlitz--Scoville formula~\eqref{eq:Carlitz} by using the following refinements  of descents and ascents. 
 
\begin{defi}
 For a permutation $\sigma\in\S_n$ with some specified boundary condition
 (i.e., specified values of $\sigma_0$ and $\sigma_{n+1}$), the index $i\in[n]$ is
 \begin{itemize}
 \item
 a {\bf\em peak} if $\sigma_{i-1}<\sigma_i>\sigma_{i+1}$;
 \item
 a {\bf\em valley} if $\sigma_{i-1}>\sigma_i<\sigma_{i+1}$;
 \item
 a {\bf\em double ascent} if $\sigma_{i-1}<\sigma_i<\sigma_{i+1}$;
 \item
 a {\bf\em double descent} if $\sigma_{i-1}>\sigma_i>\sigma_{i+1}$.
 \end{itemize}
 Let $\M(\sigma)$, $\V(\sigma)$, $\da(\sigma)$ and $\dd(\sigma)$ 
 (resp., $\widetilde\M(\sigma)$, $\widetilde\V(\sigma)$, $\tilde\da(\sigma)$ and $\tilde\dd(\sigma)$)
  denote respectively the number of  peaks, valleys, double ascents and double descents in $\sigma$ with boundary  condition $\sigma_0=\sigma_{n+1}=0$ (resp.,~$\sigma_0=\sigma_{n+1}=+\infty$). Note that $\M(\sigma)=\V(\sigma)+1$ and 
  $$
  \sum_{\sigma  \in \S_{n}}u_1^{\V(\sigma )}u_2^{\M(\sigma )}u_3^{\da(\sigma )}u_4^{\dd(\sigma )}=\sum_{\sigma  \in \S_{n}}u_1^{\widetilde\M(\sigma )}u_2^{\widetilde\V(\sigma )}u_3^{\tilde\dd(\sigma )}u_4^{\tilde\da(\sigma )}.
  $$
 \end{defi}


Ji~\cite{Ji} introduced a refinement of $A_n(x,y|\alpha,\beta)$ as 
$$
P_n(u_1,u_2,u_3,u_4|\a,\b):=\sum_{\sigma  \in \S_n}u_1^{\V(\sigma )}u_2^{\M(\sigma )-1}u_3^{\da(\sigma )}u_4^{\dd(\sigma )}\alpha^{\lma(\sigma)-1}\beta^{\rma(\sigma)-1}
$$
and proved the following generalization of~\eqref{eq:Carlitz} using Chen's context-free grammar. 

\begin{thm}[Ji~\text{\cite[Thm.~1.4]{Ji}}]
\label{Jires}
We have
  \begin{equation*}\label{Jires-equation}
  \sum_{n\geq 0}P_{n+1}(u_1,u_2,u_3,u_4|\a,\b)\frac{t^n}{n!}=\left(1+yF(x,y;t)\right)^\frac{{\alpha}+{\beta}}{2}\left(1+xF(x,y;t)\right)^\frac{{\alpha}+{\beta}}{2}e^{\frac{1}{2}({\beta}-{\alpha})(u_4-u_3)t},
\end{equation*}
  where $x+y=u_3+u_4$ and $xy=u_1u_2$.
 \end{thm}
When $\a=\b=1$, Theorem~\ref{Jires}  reduces to another  generating function formula of  Carlitz and Scoville~\cite{CS}, 
 which is equivalent to
\begin{equation}\label{carl:2}
\sum_{n\geq1}\sum_{\sigma  \in \S_{n}}u_1^{\V(\sigma )}u_2^{\M(\sigma )-1}u_3^{\da(\sigma )}u_4^{\dd(\sigma )}\frac{t^n}{n!}=F(x,y;t),\\
\end{equation}
where $x+y=u_3+u_4$ and $xy=u_1u_2$. A grammatical proof of~\eqref{carl:2} was provided by Fu~\cite{Fu}.


Let us enter the world of $q$-series. For an integer $n\geq0$, let $[n]_q!:=\prod_{i=1}^n(1+q+\cdots+q^{i-1})$ be the $q$-analog of $n!$. Introduce a classical $q$-analog of the exponential function by
$$
\exp_q(t):=\sum_{n\geq0}\frac{t^n}{[n]_q!}.
$$
Stanley (see~\cite[Sec.~6.4]{Pet})  proved the following $q$-analog of~\eqref{euler} 
  \begin{equation}\label{stanley}
 1+\sum_{n\geq1} \sum_{\sigma\in\S_n}x^{\des(\sigma)+1}q^{\inv(\sigma)}\frac{t^n}{[n]_q!}=\frac{1-x}{1-x\exp_q(t(1-x))},
 \end{equation}
where $\inv(\sigma):=|\{(i,j)\in[n]^2: i<j, \sigma_i>\sigma_j\}|$ is the {\em number of inversions} of $\sigma$. Pan and Zeng~\cite{PZ}  proved (see also~\cite{Z17,DL} for alternative proofs) the following $q$-analog of~\eqref{carl:2}, which refines Stanley's formula~\eqref{stanley}.
\begin{thm}[Pan and Zeng]\label{thm:PZ}We have
\begin{equation*}
   \sum_{n\geq1}\sum_{\sigma  \in \S_{n}}u_1^{\widetilde\V(\sigma )}u_2^{\widetilde\M(\sigma )}u_3^{\tilde\da(\sigma )}u_4^{\tilde\dd(\sigma )}q^{\inv(\sigma )}\frac{t^{n}}{[n]_q!}=F(x,y,u_4,q;t),
\end{equation*}
where 
\begin{equation}\label{def:F}
F(x,y,u,q;t):=u_1\frac{\exp_q((x-u)t)-\exp_{q}((y-u)t)}{x\exp_{q}{\left((y-u)t\right)}-y\exp_{q}{\left((x-u)t\right)}},
\end{equation}
 $x+y=u_3+u_4$ and $xy=u_1u_2$. 
\end{thm}


Inspired by these work mentioned above, it is natural to ask for a $q$-analog of Theorem~\ref{Jires}. 
For this purpose, we introduce the $q$-analog of $P_n(u_1,u_2,u_3,u_4|\a,\b)$ by
$$
P_n(u_1,u_2,u_3,u_4|\a,\b;q):=\sum_{\sigma  \in \S_n}u_1^{\V(\sigma )}u_2^{\M(\sigma )-1}u_3^{\da(\sigma )}u_4^{\dd(\sigma )}\alpha^{\lma(\sigma)-1}\beta^{\rma(\sigma)-1}q^{\inv(\sigma)}. 
$$
Our main result is the following $q$-analog of Theorem~\ref{Jires}, which will be proved via  Gessel's $q$-compositional formula~\cite{Ge}. 
  \begin{thm}\label{main}
  Let $F(x,y,u,q;t)$ be defined in~\eqref{def:F}. Then
 \begin{equation}
\sum_{n\geq 0}P_{n+1}(u_1,u_2,u_3,u_4|\a,\b;q)\frac{t^{n}}{[n]_{q}!}=\prod_{k=0}^{\infty}G_k(x,y,u_1,u_2,u_3,u_4;t),
  \end{equation} 
  where $x+y=u_3+u_4$, $xy=u_1u_2$ and 
  \begin{align*}
   G_k(x,y,u_1,u_2,u_3,u_4;t)&:=\Big[1-t\a q^k(1-q)(u_3 +u_2 F(x,y,u_4,q;tq^{k+1}))\Big]^{-1}\\
   &\qquad\times\Big[1- \frac{t\b}{q^k}(q-1)(u_4 +u_2 F(x,y,u_3,\frac{1}{q};\frac{t}{q^{k}}))\Big]^{-1}.
 \end{align*}
  \end{thm}
We will show that Theorem~\ref{main} reduces to Theorem~\ref{Jires} when $q\to 1$. As shown by Ji~\cite{Ji}, Theorem~\ref{Jires} has many nice applications, one of which is the $\gamma$-positivity of the Stirling-Eulerian polynomials $A_n(x,y|\a,\b)$ that we recall below. 

If $h(x,y)$ is symmetry and homogeneous of degree $n$, then it can be expanded as 
$$
h(x,y)=\sum_{k=0}^{\lfloor n/2\rfloor} \gamma_k(xy)^k{(x+y)}^{n-2k}.
$$
Moreover, if $\gamma_k\geq0$ for all $k$, then $h(x,y)$ is said to be {\em $\gamma$-positive}. One of the typical examples arising from
permutation statistics due to Foata and Sch\"uzenberger is the {\em bivariate Eulerian polynomials }$A_n(x,y):=\sum_{\sigma\in\S_n}x^{\asc(\sigma)}y^{\des(\sigma)}$; see the survey of Athanasiadis~\cite{Ath} for more information. 

 Using Theorem~\ref{Jires}, Ji proved the following refined $\gamma$-positivity expansion.
 \begin{thm}[\text{Ji~\cite[Theorem~1.9]{Ji}}]
 Let $\P_{n,k}$ be the set of $\sigma\in\S_n$ such that $\widetilde\M(\sigma)=k$. Then
\begin{equation}\label{g:ab}
A_n(x,y|\frac{\a+\b}{2},\frac{\a+\b}{2})=\sum_{k=0}^{\lfloor(n-1)/2\rfloor}\gamma_{n,k}(\a,\b)(xy)^k(x+y)^{n-2k-1},
\end{equation}
where 
\begin{equation*}
\gamma_{n,k}(\a,\b)=2^{2k+1-n}\sum_{\sigma\in\P_{n,k}}{\alpha}^{\lmin(\sigma)-1}\b^{\rmin(\sigma)-1}. 
\end{equation*}
\end{thm}
On the other hand, Ji and Lin~\cite{JL} proved the $\a=\b$ case combinatorially via introducing a group action on permutations, namely, 
\begin{equation}\label{g:a=b}
A_n(x,y|\a,\a)=\sum_{k=0}^{\lfloor(n-1)/2\rfloor}\gamma_{n,k}(\a)(xy)^k(x+y)^{n-2k-1},
\end{equation}
where 
\begin{equation*}
\gamma_{n,k}(\a)=2^{2k+1-n}\sum_{\sigma\in\P_{n,k}}{\alpha}^{\lmin(\sigma)+\rmin(\sigma)-2}. 
\end{equation*}
Comparing~\eqref{g:a=b} with~\eqref{g:ab} we get the following identity. 

\begin{thm}\label{pk:LR}
For $n\geq1$, we have 
\begin{equation}\label{2pnk}
\sum_{\sigma\in\P_{n,k}}{(2\a)}^{\lmin(\sigma)-1}{(2\b)}^{\rmin(\sigma)-1}=\sum_{\sigma\in\P_{n,k}}{(\a+\b)}^{\lmin(\sigma)+\rmin(\sigma)-2}.
\end{equation}
\end{thm}
We will introduce a group action on permutations to prove Theorem~\ref{pk:LR} combinatorially, which provides a combinatorial approach to~\eqref{g:ab} in view of the group action proof of~\eqref{g:a=b} in~\cite{JL}.  

The rest of this paper is organized as follows. In Section~\ref{Sec:2}, we recall a $q$-analog of the exponential formula from~\cite{Ge} and use it to prove Theorem \ref{main}. In Section~\ref{Sec:3}, we show that Theorem~\ref{main} can be described as an integral formula which reduces to Theorem~\ref{Jires} when $q\rightarrow 1$. The combinatorial proof of Theorem~\ref{pk:LR} is proved in Section~\ref{Sec:4} via introducing a new group action on permutations. 

\section{Gessel's $q$-compositional formula and proof of Theorem~\ref{main}}

 \label{Sec:2}
\subsection{Gessel's $q$-compositional formula}

 Throughout this paper, let $R$ be a ring of polynomials in some finite indeterminates  with rational coefficients. The {\em Eulerian differential operator} $\delta_t$ (see~\cite{GA})  is defined by 
$$
  \delta_t(f(t))=\frac{f(qt)-f(t)}{(q-1)t}
$$
  for any function $f(t)\in R[[t]]$  in the ring of formal power series in $t$ over $R$. Note that for $q=1$, $\delta_t$ reduces to the ordinary derivative. For example, we have 
  \begin{equation}\label{eq:det}
  \delta_t(t^n/[n]_q!)=t^{n-1}/[n-1]_q!\quad\text{for $n\geq1$}.
  \end{equation}
   For the sake of convenience, we shall often  write $f'$ for $\delta_t(f)$. 
 
 Following Gessel~\cite{Ge}, we define the following $q$-analog of the map $f\mapsto f^k/k!$ for exponential generating functions. 
 \begin{defi}
Suppose that $f(0)=0$. Define $f^{[k]}$ by $f^{[0]}=1$ and for $k\geq1$,
\begin{equation}\label{Ge:1}
\delta_t(f^{[k]})=f'\cdot f^{[k-1]} \quad\text{with $f^{[k]}(0)=0$.}
\end{equation}
\end{defi}
If we write $f^{[k]}(t)=\sum_{n\geq0}f_n^{(k)}\frac{t^n}{[n]_q!}$ for $k\geq1$, then Eq.~\eqref{Ge:1} is equivalent to the recursion 
$$
f_{n+1}^{(k)}=\sum_{i=0}^n{n\brack i}_qf_{n-i+1}^{(1)} f_{i}^{(k-1)},
$$ 
 where ${n\brack i}_q:=\frac{[n]_q!}{[i]_q![n-i]_q!}$ are the {\em$q$-binomial coefficients}.  
\begin{defi}
  Suppose that $f(0)=0$ and $g(t)=\sum_{n\geq 0}g_n\frac{t^n}{[n]_q!}$. Define the $q$-composition $g[f]$ by
  $$g[f]:=\sum_{n\geq 0}g_nf^{[n]}.$$
  \end{defi}
Note that $t^{[k]}=t^k/[k]_q!$ and so $g[t]=g(t)$.
Gessel~\cite{Ge} proved the following $q$-compositional formula, which express $\exp_q[f]$ as an infinite product.
 \begin{thm}[Gessel~\cite{Ge}]\label{Pro-Ge}
  Suppose that $f(0)=0$. Then 
  \begin{equation}\label{ef}
  \exp_q[f]=\prod_{k=0}^{\infty}\big(1-tq^k(1-q)f'(q^kt)\big)^{-1}.
  \end{equation} 
  \end{thm}
 
 A permutation is {\em basic} if it begins with its greatest element.  Denote by $\B_n$ the set of all basic permutations in $\S_n$.
  For a permutation $\sigma$ with $\lma(\sigma)=k$, there exists a unique decomposition of $\sigma$ into basic components as
\begin{equation}\label{Ge:fac}
\sigma=\b_1\b_2\cdots\b_k,
\end{equation}
where 
\begin{itemize}
\item  each $\b_i$  is basic;
\item and the first letter of each $\b_i$ is a left-to-right maximum  of $\sigma$.
\end{itemize}
Such a decomposition of $\sigma$ is called the {\em basic decomposition} of $\sigma$ in~\cite{Ge}.
 For example, $\sigma=2164573$ has the  basic decomposition $\b_1=21$, $\b_2=645$ and $\b_3=73$.
 
 Given any permutation $\sigma$ with $n$ distinct letters, let $\red(\sigma)$ be the permutation in $\S_n$ obtained from $\sigma$ by replacing  its $i$-th smallest element by $i$ for $1\leq i\leq n$. For example, $\red(8425)=4213$. A function $\omega$ from the set of all permutations to a commutative algebra over the rationals is called {\em multiplicative} if 
for all permutations $\sigma$:
 \begin{enumerate}
  \item[(i)] $\omega(\sigma)=\omega(\red(\sigma))$;
  \item[(ii)] If $\beta_1\beta_2\cdots\beta_k$ is the basic decomposition of $\sigma$, then
  $\omega(\sigma)=\omega(\beta_1)\omega(\beta_2)\cdots\omega(\beta_k).$
 \end{enumerate}
For instance, the function $\omega(\sigma)=x^{\lma(\sigma)}$ is multiplicative. 
 \begin{thm}[Gessel~\cite{Ge}]\label{Ges}
  Let $\omega$ be a multiplicative function. Suppose $g_n=\sum\limits_{\sigma\in \S_n}\omega(\sigma)q^{\inv(\sigma)}$ and
   $f_n=\sum\limits_{\beta\in \B_n}\omega(\beta)q^{\inv(\beta)}.$ Then
  \begin{equation}\label{gf}
    \sum_{n\geq 0}g_n\frac{t^n}{[n]_q!}=\exp_q\bigg[  \sum_{n\geq 1}f_n\frac{t^n}{[n]_q!}  \bigg].
  \end{equation}
  \end{thm}


\subsection{Proof of Theorem~\ref{main}}
\begin{defi}
For $\sigma\in\S_n$, let $\M_0(\sigma)$, $\V_0(\sigma)$, $\da_0(\sigma)$ and $\dd_0(\sigma)$ 
 (resp., $\M_{\infty}(\sigma)$, $\V_{\infty}(\sigma)$, $\da_{\infty}(\sigma)$ and $\dd_{\infty}(\sigma)$)
 denote respectively  the numbers of  peaks, valleys, double ascents and double descents in $\sigma$ with boundary  conditions $\sigma_0=0$ and $\sigma_{n+1}=+\infty$ (resp.,~$\sigma_0=+\infty$ and $\sigma_{n+1}=0$). Introduce two weighted  functions by 
 $$
 \omega_0(\sigma)=u_1^{\V_0(\sigma )}u_2^{\M_0(\sigma )}u_3^{\da_0(\sigma )}u_4^{\dd_0(\sigma )}\alpha^{\lma(\sigma)}
 $$
 and 
 $$
 \omega_{\infty}(\sigma)=u_1^{\V_{\infty}(\sigma )}u_2^{\M_{\infty}(\sigma )}u_3^{\da_{\infty}(\sigma )}u_4^{\dd_{\infty}(\sigma )}\b^{\rma(\sigma)}.
 $$
  \end{defi}
  The following observation can be checked routinely. 
  \begin{lem}\label{omega0}
  The function $\omega_0$ is multiplicative. 
  \end{lem}
  
  We are now in position to prove Theorem~\ref{main}.
  \begin{proof}[{\bf Proof of Theorem~\ref{main}}] Throughout the proof, we assume that $x+y=u_3+u_4$ and $xy=u_1u_2$. 
  Let 
$$
L_n(u_3,u_4,\a,q):=\sum_{\sigma\in\S_n}\omega_0(\sigma)q^{\inv{(\sigma})}\quad\text{and}\quad B_n(\a,q):=\sum_{\sigma\in\B_n}\omega_0(\sigma)q^{\inv{(\sigma})}.
$$
It follows from Lemma~\ref{omega0} and Theorem~\ref{Ges} that 
\begin{equation}\label{eq:omega}
  \sum_{n\geq 0}L_n(u_3,u_4,\alpha,q)\frac{t^n}{[n]_q!}=\exp_q\bigg[\sum_{n\geq 1}B_n(\alpha,q)\frac{t^n}{[n]_q!}\bigg].
\end{equation}
As $B_1(\a,q)=u_3\alpha$ and for $n\geq2$,
$$
B_n(\alpha,q)=u_2\alpha q^{n-1}\sum_{\sigma\in\S_{n-1}}u_1^{\widetilde\V(\sigma)}u_2^{\widetilde\M(\sigma)} u_3^{\tilde\da(\sigma)}u_4^{\tilde\dd(\sigma)}q^{\inv(\sigma)},
$$
by~\eqref{eq:det} we have
\begin{align*}
\delta_t\biggl(\sum_{n\geq 1}B_n(\alpha,q)\frac{t^n}{[n]_q!}\biggr)&=
\sum_{n\geq 1}B_n(\alpha,q)\frac{t^{n-1}}{[n-1]_q!}\\
&=u_3\alpha +u_2\alpha\sum_{n\geq 1} \sum_{\sigma\in\S_{n}}u_1^{\widetilde\V(\sigma)}u_2^{\widetilde\M(\sigma)} u_3^{\tilde\da(\sigma)}u_4^{\tilde\dd(\sigma)}q^{\inv(\sigma)}\frac{(tq)^{n}}{[n]_q!}\\
&=u_3\alpha+u_2\alpha F(x,y,u_4,q;tq)
\end{align*}
in view of Theorem~\ref{thm:PZ}. It then follows from Theorem~\ref{Pro-Ge} and Eq.~\eqref{eq:omega} that 
\begin{equation}\label{Ln}
  \sum_{n\geq 0}L_n(u_3,u_4,\alpha,q)\frac{t^n}{[n]_q!}=\prod_{k\geq 0}\Big[1-t\a q^k(1-q)(u_3 +u_2 F(x,y,u_4,q;tq^{k+1}))\Big]^{-1}.
\end{equation}

On the other hand, we need to compute the exponential generating function for 
$$
  R_n(u_3,u_4,\beta,q):=\sum_{\sigma\in\S_n}\omega_{\infty}(\sigma)q^{\inv{(\sigma})}.
 $$
For $\sigma\in\S_n$,  let $\sigma^r:=\sigma_n\sigma_{n-1}\cdots\sigma_1$ be the {\em reversal} of $\sigma$.  As 
$$\V_{\infty}(\sigma^r)=\V_0(\sigma ),~\M_{\infty}(\sigma^r)=\M(\sigma),~\da_{\infty}(\sigma^r)=\dd_0(\sigma),~\dd_{\infty}(\sigma^r)=\da_0(\sigma )
$$
and
$$\rma(\sigma^r)=\lma(\sigma),\quad\inv(\sigma^r)=\binom{n}{2}-\inv{(\sigma}),
$$ we have
\begin{align*}
\sum_{n\geq 0} R_n(u_3,u_4,\beta,q)\frac{t^n}{[n]_q!}&=\sum_{n\geq 0}q^{\binom{n}{2}}L_n(u_4,u_3,\beta,q^{-1})\frac{t^n}{[n]_q!}\\
&=\sum_{n\geq 0}L_n(u_4,u_3,\beta,q^{-1})\frac{t^n}{[n]_{q^{-1}}!}.
\end{align*}
It then follows from~\eqref{Ln}  that
\begin{equation}\label{Rn}
  \sum_{n\geq 0}R_n(u_3,u_4,\b,q)\frac{t^n}{[n]_q!}=\prod_{k\geq 0}\Big[1- \frac{t\b}{q^k}(1-\frac{1}{q})(u_4 +u_2 F(x,y,u_3,\frac{1}{q};\frac{t}{q^{k+1}}))\Big]^{-1}.
\end{equation}

It is known~\cite[Prop~1.3.17]{St} that   ${n\brack k}_q$
has the following interpretation
$$
{n\brack  k}_{q}=\sum_{({\mathcal A}, {\mathcal B})}q^{\inv({\mathcal A}, {\mathcal B})},
$$
 summed  over all ordered partitions $({\mathcal A}, {\mathcal B})$ of $[n]$ such that $|{\mathcal A}|=k$. Let $$\S_{n+1,k+1}:=\{\sigma\in\S_{n+1}:\sigma_{k+1}=n+1\}.$$
As  any permutation $\sigma\in\S_{n+1,k+1}$  can be written as $\sigma=\sigma_L(n+1)\sigma_R$ with $\sigma_L=\sigma_1\cdots\sigma_k$ and $\sigma_R=\sigma_{k+2}\cdots\sigma_{n+1}$,  we have 
\begin{align*}
&\quad\sum_{\sigma\in\S_{n+1,k+1}}u_1^{\V(\sigma )}u_2^{\M(\sigma )-1}u_3^{\da(\sigma )}u_4^{\dd(\sigma )}\alpha^{\lma(\sigma)-1}\beta^{\rma(\sigma)-1}q^{\inv(\sigma )}\\
&={n\brack k}_q q^{n-k}L_k(u_3,u_4,\alpha,q)R_{n-k}(u_3,u_4,\b,q)
\end{align*}
in view of  the above interpretation of ${n\brack k}_q$. Thus, 
$$
P_{n+1}(u_1,u_2,u_3,u_4|\a,\b;q)=\sum_{k=0}^n{n\brack k}_q q^{n-k}L_k(u_3,u_4,\alpha,q)R_{n-k}(u_3,u_4,\b,q),
$$
which in terms of exponential generating function gives
$$
\sum_{n\geq0}P_{n+1}(u_1,u_2,u_3,u_4|\a,\b;q)\frac{t^n}{[n]_q!}=\bigg(\sum_{n\geq0}L_n(u_3,u_4,\a,q)\frac{t^n}{[n]_q!}\bigg)\bigg(\sum_{n\geq0}R_n(u_3,u_4,\b,q)\frac{(tq)^n}{[n]_q!}\bigg).
$$
Substituting~\eqref{Ln} and~\eqref{Rn} into the above identity proves  Theorem \ref{main}.
\end{proof}

\section{A new proof of Theorem~\ref{Jires}}
\label{Sec:3}

 For any $f(t)\in R[[t]]$,  define the {\em$q$-integral} (see~\cite[Section~1.11]{GR}) of $f$ by
\begin{equation}\label{int}
  \int_{0}^{t}f(x)d_qx:=t(1-q)\sum_{n\geq 0}f(q^n t)q^n.
\end{equation}
As $ \lim\limits_{q\rightarrow 1}\int_{0}^{t}x^kd_qx=t^{k+1}/(k+1)$ for $k\geq0$, we have
\begin{equation}\label{qtonq}
 \lim_{q\rightarrow 1} \int_{0}^{t}f(x)d_qx=\int_{0}^{t}f(x)dx.
\end{equation}

Our main result Theorem~\ref{main} can also  be  described as an integral formula.
\begin{thm}\label{main2}
 We have
  \begin{align*}
 &\quad\sum_{n\geq 0}P_{n+1}(u_1,u_2,u_3,u_4|\a,\b;q)\frac{t^{n}}{[n]_{q}!}\\
  &=\exp\Big(\int_{0}^{t}\frac{\ln\big[1-(1-q)z\a\big(u_3 +u_2 F(zq)\big)\big]^{-1}}{(1-q)z} d_q z 
     +\int_{0}^{t}\frac{\ln\big[1-(q-1)z\b\big(u_4 +u_2 \tilde F(z)\big)\big]^{-1}}{(1-\frac{1}{q})z} d_{\frac{1}{q}} z\Big),
  \end{align*} 
  with $x+y=u_3+u_4$,~$xy=u_1u_2$,
  \begin{equation}\label{Def:FtF}
F(t):=F(x,y,u_4,q;t)\quad\text{and}\quad\tilde F(t):=F(x,y,u_3,\frac{1}{q};t).
  \end{equation}
\end{thm}

\begin{proof}
As 
$\ln(1-t)^{-1}=\sum_{n\geq1}\frac{t^n}{n}$,
 we have
\begin{align*}
  &\quad\prod_{k\geq 0}\Big[1-t\a q^k(1-q)(u_3 +u_2 F(tq^{k+1}))\Big]^{-1}\\
  &=\exp\bigg(\ln\Big(\prod_{k\geq 0}[1-t\a q^k(1-q)(u_3 +u_2 F(tq^{k+1}))]^{-1}\Big)\bigg)\\
  &=\exp\bigg(\sum_{k\geq 0}\ln[1-t\a q^k(1-q)(u_3 +u_2 F(tq^{k+1}))]^{-1}\bigg)\\
  &=\exp\bigg(\sum_{k\geq 0}\sum_{n\geq 1}\frac{\big[t\a q^k(1-q)(u_3 +u_2 F(tq^{k+1}))\big]^n}{n}\bigg)\\
  &=\exp\bigg(\sum_{n\geq 1}\frac{(1-q)^{n-1}}{n}(1-q)t\sum_{k\geq 0}q^k (tq^k )^{n-1}(u_3\alpha +u_2\alpha F(tq^{k+1}))^n\bigg)\\
  &=\exp\bigg( \sum_{n\geq 1}\frac{(1-q)^{n-1}}{n} \int_{0}^{t} z^{n-1}\big(u_3\alpha +u_2\alpha F(zq)\big)^n d_q z\bigg)\qquad\text{By Def.~\eqref{int}}\\
  &=\exp\bigg(\int_{0}^{t}\sum_{n\geq 1}\frac{\big[(1-q)z\big(u_3\alpha +u_2\alpha F(zq)\big)\big]^n}{n(1-q)z} d_q z\bigg)\\
  &=\exp\bigg(\int_{0}^{t}\frac{\ln\big[1-(1-q)z\a\big(u_3 +u_2 F(zq)\big)\big]^{-1}}{(1-q)z} d_q z\bigg).
\end{align*}
Similarly, we have 
$$
\prod_{k\geq 0}\Big[1-\frac{t\b}{q^k} (q-1)(u_4 +u_2 \tilde F(tq^{-k}))\Big]^{-1}=\exp\bigg(\int_{0}^{t}\frac{\ln\big[1-(q-1)z\b\big(u_4 +u_2 \tilde F(z)\big)\big]^{-1}}{(1-\frac{1}{q})z} d_{\frac{1}{q}} z\bigg).
$$
The desired integral formula then follows from Theorem~\ref{main}. 
\end{proof}

 Next, we  show that our integral formula in Theorem~\ref{main2} reduces to Theorem~\ref{Jires} when $q\rightarrow1$. We need the following simple lemma. 
 \begin{lem}
   Let $F(x,y;t)$ be defined in~\eqref{def:Fxy}. Then, 
  \begin{equation}\label{parint}
  \int_{0}^{t}xy F(x,y;z)dz=\ln\Big(\frac{x-y}{x e^{yt}-y e^{xt}}\Big).
  \end{equation}
\end{lem}

  \begin{proof}
We have
\begin{equation*}
   \frac{d\big(-\ln[x e^{yt}-y e^{xt}]\big)}{dt}=\frac{xy(e^{xt}-e^{yt})}{x e^{yt}-y e^{xt}}=xyF(x,y;t),
  \end{equation*} 
from which the result then follows. 
\end{proof}

We are now ready for a new proof of Theorem~\ref{Jires}. 

\begin{proof}[{\bf A new proof of Theorem~\ref{Jires}.}] 
When $q\to 1$, the two infinitesimal quantities 
$$\ln\big[1-(1-q)t\a\big(u_3 +u_2 F(tq)\big)\big]\quad\text{and}\quad (1-q)t\a\big(u_3+u_2 F(tq)\big)
$$ are equivalent, and so are  
$$\ln\big[1-(q-1)t\b\big(u_4+u_2 \tilde{F}(t)\big)\big]\quad\text{and}\quad (q-1)t\b\big(u_4 +u_2 \tilde F( t)\big),
$$
where $F(t)$ and $\tilde F(t)$ are defined in~\eqref{Def:FtF}. Thus, if we let $q\to 1$ in Theorem~\ref{main2}, then by~\eqref{qtonq} we have 
\begin{align*}
&\quad\sum_{n\geq 0}P_{n+1}(u_1,u_2,u_3,u_4|\a,\b)\frac{t^n}{n!}\\
&=\exp\Big(\int_{0}^{t}\a\big(u_3 +u_2 F(x,y,u_4,1;z)\big)dz +\int_{0}^{t}\b\big(u_4 +u_2 F(x,y,u_3,1;z)\big)dz\Big)\\
&=\exp\Big(\int_{0}^{t}\a\big(u_3 +xy F(x,y;z)\big)dz +\int_{0}^{t}\b\big(u_4 +xy F(x,y;z)\big)dz\Big)
\end{align*} with $xy=u_1u_2$ and $x+y=u_3+u_4$. Using~\eqref{parint} we get 
$$
\sum_{n\geq 0}P_{n+1}(u_1,u_2,u_3,u_4|\a,\b)\frac{t^n}{n!}=e^{(u_3\alpha+u_4\beta)t}\Big(\frac{x-y}{xe^{yt}-ye^{xt}}\Big)^{\a+\b},
$$
completing the proof. 
\end{proof}

\section{Combinatorics of the $\gamma$-positivity of Stirling-Eulerian polynomials}
\label{Sec:4}

This section aims to prove Theorem~\ref{pk:LR} combinatorially by introducing a new group action on permutations.  

For $\sigma\in\S_n$ with $\lmin(\sigma)-1=k$ and $\rmin(\sigma)-1=l$, there exists a unique decomposition of $\sigma$ as
$$
\sigma=\a_1\a_2\cdots\a_k1\b_1\b_2\cdots\b_l,
$$
where 
\begin{itemize}
\item the first letter of each $\a_i$ ($1\leq i\leq k$), which is smallest  inside $\a_i$, is a left-to-right minimum of $\sigma$; 
\item  and the last letter of each $\b_i$ ($1\leq i\leq l$), which is smallest  inside $\b_i$, is a right-to-left minimum of $\sigma$. 
\end{itemize}
We call this decomposition the {\em bi-basic decomposition }of $\sigma$ and each $\a_i$ or $\b_i$ a block of $\sigma$. An example of the bi-basic decomposition  is 
\begin{equation}\label{bi-basic}
\sigma=\magenta{{\bf 5}\,10}|\magenta{{\bf2}\,12\,4\,13\,6}|{\bf1}|\blue{11\,{\bf3}}|\blue{9\,8\,15\,{\bf7}}|\blue{{\bf14}},
\end{equation}
where blocks are separated by bars. For each word $w=w_1w_2\cdots w_d$, let $w^r:=w_dw_{d-1}\cdots w_1$ be the reversal of $w$.  For each $x\in[n]$, we define $\psi_x(\sigma)$ according to the following cases
\begin{enumerate}
\item if $x$ is the first letter of  $\a_i$ for some $i$, then $\psi_x(\sigma)$ is obtained from $\sigma$ by deleting the block $\a_i$ and then inserting $\a_i^r$ in some proper gap between blocks so that $x$ becomes an additional right-to-left minimum;  
\item if $x$ is the first letter of  $\b_i$ for some $i$, then $\psi_x(\sigma)$ is obtained from $\sigma$ by deleting the block $\b_i$ and then inserting $\b_i^r$ in some proper gap between blocks so that $x$ becomes an additional left-to-right minimum;  
\item otherwise, set $\psi_x(\sigma)=\sigma$. 
\end{enumerate}
Taking $\sigma$ as in~\eqref{bi-basic}, then 
$$
\psi_2(\sigma)=\magenta{{\bf 5}\,10}|{\bf1}|\magenta{6\,13\,4\,12\,{\bf2}}|\blue{11\,{\bf3}}|\blue{9\,8\,15\,{\bf7}}|\blue{{\bf14}},
$$
and 
$$
\psi_3(\sigma)=\magenta{{\bf 5}\,10}|\blue{{\bf3}\,11}|\magenta{{\bf2}\,12\,4\,13\,6}|{\bf1}|\blue{9\,8\,15\,{\bf7}}|\blue{{\bf14}}. 
$$
It is clear that the map $\psi_x$ is an involution acting on $\S_n$, and for all $x, y \in [n]$, $\psi_x$ and $\psi_y$ commute. The following lemma is evident from  the construction of $\psi_x$. 

\begin{lem}\label{lem:psi}
If $x$ is a left-to-right (resp., right-to-left) minimum of $\sigma$, then $x$ becomes a right-to-left (resp., left-to-right) minimum of $\psi_x(\sigma)$. Moreover, $\psi_x$ preserves the peak statistic 
$``\widetilde\M"$.
\end{lem}

We are now ready for the combinatorial proof of Theorem~\ref{pk:LR}.
\begin{proof}[{\bf Proof of Theorem~\ref{pk:LR}}]
Denote by $\LMI(\sigma)$ and $\RMI(\sigma)$ the set of left-to-right minima and the set of right-to-left minima of $\sigma$ other than $1$, respectively. 
In order to interpret the left-hand side of~\eqref{2pnk}, we introduce the set 
$$\widetilde\P_{n,k}=\{(\sigma,S): \sigma\in\P_{n,k}, S\subseteq\LMI(\sigma)\cup\RMI(\sigma)\}
$$ of marked permutations, where $\P_{n,k}:=\{\sigma\in\S_n|\widetilde\M(\sigma)=k\}$.  Thus, 
\begin{equation}\label{int:mark}
\sum_{\sigma\in\P_{n,k}}{(2\a)}^{\lmin(\sigma)-1}{(2\b)}^{\rmin(\sigma)-1}=\sum_{(\sigma,S)\in\widetilde\P_{n,k}}\a^{\lmin(\sigma)-1}\b^{\rmin(\sigma)-1}.
\end{equation}
For $x\in[n]$ and $(\sigma,S)\in\widetilde\P_{n,k}$, we define the action $\Psi_x$ on $\widetilde\P_{n,k}$ as   
$$\Psi_x(\sigma,S)=
\begin{cases}
(\psi_x(\sigma),S\triangle\{x\}),\quad &\text{if $x\in\LMI(\sigma)\cup\RMI(\sigma)$};\\
(\sigma,S), &\text{otherwise}. 
\end{cases}
$$ 
It is routine to check that all $\Psi_x$'s are involutions and commute. Therefore, for any $X\subseteq[n]$ we can define the function $\Psi_X:\widetilde\P_{n,k}\rightarrow\widetilde\P_{n,k}$ by $\Psi_X=\prod_{x\in X}\Psi_x$, where multiplication is the composition of functions. Hence the group $\Z_2^n$ acts on $\widetilde\P_{n,k}$ via the function $\Psi_X$. Let $\Orb(\sigma,S)$ be the orbit of $(\sigma,S)$ under this action. Inside each orbit $\Orb(\sigma,S)$ there exists a unique element $(\bar\sigma,\emptyset)$, where $\bar\sigma=\Psi_S(\sigma)$. It then follows from Lemma~\ref{lem:psi} that 
$$
\sum_{(\pi,T)\in\Orb(\sigma,S)}\a^{\lmin(\pi)-1}\b^{\rmin(\pi)-1}=(\a+\b)^{\lmin(\bar\sigma)+\rmin(\bar\sigma)-2}. 
$$
Summing over all orbits of $\widetilde\P_{n,k}$ then results in
$$
\sum_{(\sigma,S)\in\widetilde\P_{n,k}}\a^{\lmin(\sigma)-1}\b^{\rmin(\sigma)-1}=\sum_{\sigma\in\P_{n,k}}{(\a+\b)}^{\lmin(\sigma)+\rmin(\sigma)-2},
$$
which proves~\eqref{2pnk} in view of~\eqref{int:mark}.  
\end{proof} 

As a byproduct of the action $\psi_x$, we can also prove the following neat expression for $\gamma_{n,k}(\a,\b)$. 
\begin{thm}\label{pk:LR2}
For $n\geq1$, we have 
\begin{equation}\label{2pnk2}
\sum_{\sigma\in\P_{n,k}}{\a}^{\lmin(\sigma)-1}{\b}^{\rmin(\sigma)-1}=\sum_{\sigma\in\L_{n-1,k}}{(\a+\b)}^{\rmin(\sigma)},
\end{equation}
where $\L_{n-1,k}=\{\sigma\in\S_{n-1}: \M_{0}(\sigma)=k\}$. Consequently, 
\begin{equation*}
\gamma_{n,k}(\a,\b)=2^{2k+1-n}\sum_{\sigma\in\L_{n-1,k}}{(\a+\b)}^{\rmin(\sigma)}. 
\end{equation*}
\end{thm}

\begin{proof}
We consider the function $\psi_X: \P_{n,k}\rightarrow\P_{n,k}$ defined by $\psi_X=\prod_{x\in X}\psi_x$. Then, the group $\Z_2^n$ acts on $\P_{n,k}$ via the function $\psi_X$. Let $\Orb(\sigma)$ be the orbit of $\sigma$ under this action. Inside each orbit $\Orb(\sigma)$ there exists a unique element $\bar\sigma$ whose first letter is $1$. It then follows from Lemma~\ref{lem:psi} that 
$$
\sum_{\pi\in\Orb(\sigma)}\a^{\lmin(\pi)-1}\b^{\rmin(\pi)-1}=(\a+\b)^{\rmin(\bar\sigma)-1}. 
$$
Summing over all orbits of $\P_{n,k}$ then results in
\begin{align*}
\sum_{\sigma\in\P_{n,k}}{\a}^{\lmin(\sigma)-1}{\b}^{\rmin(\sigma)-1}&=\sum_{\sigma\in\P_{n,k}\atop \sigma_1=1}(\a+\b)^{\rmin(\bar\sigma)-1}\\
&=\sum_{\sigma\in\L_{n-1,k}}{(\a+\b)}^{\rmin(\sigma)},
\end{align*}
where the last equality follows from the simple one-to-one correspondence 
$$
\sigma\mapsto (\sigma_2-1)(\sigma_3-1)\cdots(\sigma_n-1)
$$ between $\{\sigma\in\P_{n,k}:\sigma_1=1\}$ and $\L_{n-1,k}$. 
\end{proof}

The {\em Euler numbers} $\{E_n\}_{n\geq0}$ can be  defined by the Taylor expansion 
 $$
 \sec(t)+\tan(t)=\sum_{n\geq0} E_n\frac{t^n}{n!}.
 $$
 The numbers $E_{n}$ with even indices and odd indices  are known as {\em secant numbers} and {\em tangent numbers}, respectively.
A permutation $\sigma\in\S_n$ is  {\em alternating} if 
$$
\sigma_1>\sigma_2<\sigma_3>\sigma_4<\sigma_5>\cdots.
$$
Andr\'e~\cite{And} proved that $|\mathcal{A}_n|=E_n$, where  $\mathcal{A}_n$ denotes the set of all alternating permutations of length $n$.     An immediate consequence of Theorem~\ref{pk:LR2} is the following interesting interpretation of the $\a$-extension of secant  numbers 
$$
  A_{2n+1}(-1,1|\a/2,\a/2)=(-1)^{n}\sum_{\sigma\in\mathcal{A}_{2n}}\a^{\rmin(\sigma)},
 $$
 which was proved by Ji~\cite{Ji} via different approach. 
 
\section{Concluding remarks}

The main achievement of this paper is the computation of  the  generating function for a $q$-analog of $P_n(u_1,u_2,u_3,u_4|\a,\b)$
by using Gessel's $q$-compositional formula.  In particular, Ji's formula for $P_n(u_1,u_2,u_3,u_4|\a,\b)$ in Theorem~\ref{Jires} can be proved directly from Carlitz and Scoville's formula~\eqref{eq:Carlitz} by using the classical  compositional formula (see~\cite[Thm.~5.1.4]{St2}). However, it should be noted that this approach via the  compositional formula does not work for the variation of $P_n(u_1,u_2,u_3,u_4|\a,\b)$: 
$$
\tilde P_n(u_1,u_2,u_3,u_4|\a,\b):=\sum_{\sigma  \in \S_n}u_1^{\widetilde\V(\sigma )-1}u_2^{\widetilde\M(\sigma )}u_3^{\tilde\da(\sigma )}u_4^{\tilde\dd(\sigma )}\alpha^{\lma(\sigma)-1}\beta^{\rma(\sigma)-1}.
$$
It remains open to compute the generating function for $\tilde P_n(u_1,u_2,u_3,u_4|\a,\b)$. 

A statistic whose distribution over $\S_n$ gives $[n]_q!$ is called a {\em Mahonian} statistic. It is well known that `$\inv$' is a  Mahonian statistic on permutations. Another classical Mahonian statistic is the {\em major index} (see~\cite[Sec.~1.3]{St})
$$
\maj(\sigma)=\sum_{\sigma_i>\sigma_{i+1}}i. 
$$
A third Mahonian statistic arising from poset topology was introduced by Shareshian and Wachs~\cite{SW} as `$\ai+\des$', where $\ai(\sigma)$ is the number of {\em admissible inversions} of $\sigma$ (see also~\cite{LZ}). The generating functions for the four pairs of Euler--Mahonian statistics in $\{\exc,\des\}\times\{\inv,\maj\}$ have already been calculated in the literature; see~\cite{SW}. In view of Theorem~\ref{main}, it would be interesting to investigate systematically the generating functions for the Stirling--Euler--Mahonian statistics on permutations in 
$$
\{\lmin,\rmin,\lma,\rma,\cyc\}\times\{\des,\exc,\M\}\times\{\inv,\maj,\ai\}
$$
and their possible  refinements. Some related results in this direction have been obtained by Sokal and Zeng~\cite{SZ}.

\section*{Statements and Declarations}
The authors declare that they have no known competing financial interests or personal relationships that could have appeared to influence the work reported in this paper.
 
 \section*{Acknowledgement}

This work was supported by the National Science Foundation of China grants 12271301, 12322115 and 12201468. Part of this work was done  while the first  author was visiting the third author at Wenzhou University in the winter of 2023.

\end{document}